\documentclass[11pt,twoside]{article}

\usepackage{mathrsfs,amsfonts,amsmath,amssymb}
\usepackage{latexsym}
\newtheorem{satz}{Theorem}
\newtheorem{proposition}[satz]{Proposition}
\newtheorem{theorem}[satz]{Theorem}
\newtheorem{lemma}[satz]{Lemma}
\newtheorem{definition}[satz]{Definition}
\newtheorem{corollary}[satz]{Corollary}

\newtheorem{example}[satz]{Example}

\def\eps{\varepsilon}
\def\_phi{\varphi}

\def\a{\alpha}
\def\d{\delta}
\def\la{\lambda}

\def\v{\vec}
\def\F{{\mathbb F}}

\def\ov{\overline}

\def\C{{\mathbb C}}
\def\R{{\mathbb R}}
\def\E{\mathsf {E}}

\def\Z_N{{\mathbb Z}_N}
\def\Z{{\mathbb Z}}
\def\N{{\mathbb N}}

\def\f{{\mathbb F}}
\def\Gr{{\mathbf G}}

\def\D{{\mathbb D}}

\def\FF{\widehat}
\def\c{\circ}
\def\D{\Delta}
\def\Cf{{\mathcal C}}
\def\wt{{\mathrm {wt}}}

\def\Sym{\mathsf {Sym}}

\author{Shkredov I.D.}
\title{ Some remarks on the Balog--Wooley decomposition theorem and quantities $D^{+}$, $D^\times$
\footnote{
This work was supported by grant
Russian Scientific Foundation RSF 14--11--00433.}
}
\date{}
\begin{document}
\maketitle

\begin{center}
 Annotation.
\end{center}

{\it \small
    In the paper we study two characteristics $D^+ (A)$, $D^\times (A)$ of a set $A \subset \R$ which play important role in recent results concerning sum--product phenomenon.
    Also we obtain several variants and improvements of the Balog--Wooley decomposition theorem. In particular, we prove that any finite subset of $\R$ can be split into two sets with small quantities $D^+$ and $D^\times$.
}
\\

\section{Introduction}
\label{sec:introduction}

Let  $A,B\subset \R$ be finite sets.
Define the  \textit{sum set}, the \textit{product set} and \textit{quotient set} of $A$ and $B$ as
$$A+B:=\{a+b ~:~ a\in{A},\,b\in{B}\}\,,$$
$$AB:=\{ab ~:~ a\in{A},\,b\in{B}\}\,,$$
and
$$A/B:=\{a/b ~:~ a\in{A},\,b\in{B},\,b\neq0\}\,,$$
correspondingly.
The Erd\"{o}s--Szemer\'{e}di  conjecture \cite{ES} says that for any  $\epsilon>0$ one has
$$\max{\{|A+A|,|AA|\}}\gg{|A|^{2-\epsilon}} \,.$$
Roughly speaking, it asserts that an arbitrary subset of real numbers (or integers)
cannot has good additive and multiplicative structure, simultaneously.
Modern bounds concerning the conjecture can be found in
\cite{soly}, \cite{KS1}, \cite{KS2}.

Define
$$
    \E^+(A,B) := |\{ a+b=a'+b' ~:~ a,a'\in A,\, b,b'\in B\} \,,
$$
and
$$
    \E^\times (A,B) := |\{ ab=a'b' ~:~ a,a'\in A,\, b,b'\in B\}
$$
be the {\it additive} and the {\it multiplicative common energies} of $A$ and $B$, correspondingly.
Numbers $\E^+(A,A)$, $\E^\times (A,A)$ are another measures to control the additivity and the multiplicativity of a set.

\bigskip

In \cite{BW} the following decomposition theorem was proved.

\begin{theorem}
    Let $A\subset \R$ be a finite set and $\d = 2/33$.
    Then there are two disjoint subsets $B$ and $C$ of $A$ such that $A = B\sqcup C$ and
\begin{equation}\label{f:BW_1}
    \max\{ \E^{+} (B,B), \E^{\times} (C,C)\} \ll |A|^{3-\delta} (\log |A|)^{1-\delta}
\end{equation}
    and
\begin{equation}\label{f:BW_2}
    \max\{ \E^{+} (B,C), \E^{\times} (B,C)\} \ll |A|^{3-\d/2} (\log |A|)^{(1-\d)/2} \,.
\end{equation}
\label{t:BW}
\end{theorem}

Here and below we suppose that $|A|\ge2$.
All logarithms are base $2.$ Signs $\ll$ and $\gg$ are the usual Vinogradov's symbols.
     We will write $a \lesssim b$ or $b \gtrsim a$ if $a = O(b \cdot \log^c |A|)$, $c>0$.
     If $a \lesssim b$ and $b \lesssim a$ then we write $a\sim b$.

The results of such type are useful, see e.g. \cite{Hanson}
and will find further applications by the author's opinion.
Also it was proved in \cite{BW} that one cannot take $\d$
greater
than $2/3$.

In \cite{KS2}  a different method was  applied and it
allows
to obtain an improvement.

\begin{theorem}
    Let $A\subset \R$ be a finite set and $\d = 1/5$.
    Then there are two disjoint subsets $B$ and $C$ of $A$ such that $A = B\sqcup C$ and
$$
    \max\{ \E^{+} (B,B), \E^{\times} (C,C)\} \lesssim  |A|^{3-\d} \,.
$$
\label{t:BW'}
\end{theorem}

Actually, a more stronger result takes place.
Write
$$
    \E^{+}_3 (A) = | \{ a_1 - a'_1 = a_2 - a'_2 = a_3 - a'_3 ~:~ a_1,a'_1,a_2,a'_2,a_3,a'_3 \in A \}|
$$
and similar for $\E^{\times}_3 (A)$.
Using the fact that $(\E^{+}_3 (f))^{1/6}, (\E^{\times}_3 (f))^{1/6}$ are norms of a real function $f$
(more precisely see section \ref{sec:norms}), we obtain

\begin{theorem}
    Let $A\subset \R$ be a finite set and $\d_1 = 2/5$.
    Then there
    are
    two disjoint subsets $B$ and $C$ of $A$ such that $A = B\sqcup C$ and
\begin{equation}\label{f:R_E_3_introduction}
    \max\{\E^{+}_3 (B), \E^{\times}_3 (C)\} \lesssim  |A|^{4-\d_1} \,,
\end{equation}
    Besides, inequality (\ref{f:R_E_3_introduction}) cannot holds with $\d_1$ greater than $3/4$.
\label{t:R_E_3_introduction}
\end{theorem}

Using the H\"{o}lder inequality
$$
    (\E^{+} (A,A))^2 \le \E^{+}_3 (A) |A|^2,\,
        \quad \quad
            (\E^{\times} (A,A))^2 \le \E^{\times}_3 (A) |A|^2 \,,
$$
it is easy to see that Theorem \ref{t:R_E_3_introduction} implies Theorem \ref{t:BW'}.

\bigskip

Actually, our proof allows to say much more about the sets $B,C$ than
 is written in
 Theorem \ref{t:R_E_3_introduction}.
 We consider two quantities $D^{+}, D^{\times}$ (see the definitions in section
 \ref{sec:preliminaries}) and prove the strongest decomposition result.

\begin{theorem}
    Let $A\subset \R$ be a finite set and $\d_2 = 2/5$.
    Then there are two disjoint subsets $B$ and $C$ of $A$ such that $A = B\sqcup C$ and
\begin{equation}\label{f:BW_D_introduction}
    \max\{ D^{+} (B), D^{\times} (C)\} \lesssim  |A|^{1-\d_2} \,.
\end{equation}
    Besides, inequality (\ref{f:BW_D_introduction}) cannot holds with $\d_2$ greater than $3/4$.
\label{t:BW_D_introduction}
\end{theorem}

It is easy to
check
that Theorem \ref{t:BW_D_introduction} implies both Theorem \ref{t:BW'}, \ref{t:R_E_3_introduction}
(see discussion in section \ref{sec:proof}).
The quantities $D^{+}, D^{\times}$ play an important role in additive combinatorics,
  see e.g. \cite{s_sumsets}, \cite{KS1}, \cite{KS2}. For example, studying the characteristics of a set allows us to improve the famous Solymosi $4/3$ result, see \cite{soly}.


Also, in section \ref{sec:proof} we
obtain
several other forms of the Balog--Wooley Theorem, study quantities $D^+ (A)$, $D^\times (A)$
and
find some applications to sum--product questions, see e.g. Theorem \ref{t:dd} below.

We are going to obtain similar results in $\F_p$ in a forthcoming paper.

The author is grateful to Sergey Konyagin
for useful discussions.

\section{Notation}
\label{sec:definitions}

Let $\Gr$ be an abelian group.
If $\Gr$ is finite then denote by $N$ the cardinality of $\Gr$.
It is well--known~\cite{Rudin_book} that the dual group $\FF{\Gr}$ is isomorphic to $\Gr$ in the case.
Let $f$ be a function from $\Gr$ to $\mathbb{C}.$  We denote the Fourier transform of $f$ by~$\FF{f},$
\begin{equation}\label{F:Fourier}
  \FF{f}(\xi) =  \sum_{x \in \Gr} f(x) e( -\xi \cdot x) \,,
\end{equation}
where $e(x) = e^{2\pi i x}$
and $\xi$ is a homomorphism from $\FF{\Gr}$ to $\R/\Z$ acting as $\xi : x \to \xi \cdot x$.
We rely on the following basic identities
\begin{equation}\label{F_Par}
    \sum_{x\in \Gr} |f(x)|^2
        =
            \frac{1}{N} \sum_{\xi \in \FF{\Gr}} \big|\widehat{f} (\xi)\big|^2 \,,
\end{equation}
\begin{equation}\label{svertka}
    \sum_{y\in \Gr} \Big|\sum_{x\in \Gr} f(x) g(y-x) \Big|^2
        = \frac{1}{N} \sum_{\xi \in \FF{\Gr}} \big|\widehat{f} (\xi)\big|^2 \big|\widehat{g} (\xi)\big|^2 \,,
\end{equation}
and
\begin{equation}\label{f:inverse}
    f(x) = \frac{1}{N} \sum_{\xi \in \FF{\Gr}} \FF{f}(\xi) e(\xi \cdot x) \,.
\end{equation}
If
\begin{equation}\label{f:convolutions}
    (f*g) (x) := \sum_{y\in \Gr} f(y) g(x-y) \quad \mbox{ and } \quad
        (f\circ g) (x) := \sum_{y\in \Gr} f(y) g(y+x)
\end{equation}
 then
\begin{equation}\label{f:F_svertka}
    \FF{f*g} = \FF{f} \FF{g} \quad \mbox{ and } \quad \FF{f \circ g} = \FF{f^c} \FF{g} = \ov{\FF{\ov{f}}} \FF{g} \,,
\end{equation}
where for a function $f:\Gr \to \mathbb{C}$ we put $f^c (x):= f(-x)$.
 Clearly,  $(f*g) (x) = (g*f) (x)$ and $(f\c g)(x) = (g \c f) (-x)$, $x\in \Gr$.
 The $k$--fold convolution, $k\in \N$  we denote by $*_k$,
 so $*_k := *(*_{k-1})$.

In the paper we use the same letter to denote a set $S\subseteq \Gr$
and its characteristic function $S:\Gr \rightarrow \{0,1\}.$
By $|S|$ denote the cardinality of $S$.
For a positive integer $n,$ we set $[n]=\{1,\ldots,n\}.$

Put
$\E^{+}(A,B)$ for the {\it additive energy} of two sets $A,B \subseteq \Gr$
(see e.g. \cite{TV}), that is
$$
    \E^{+} (A,B) = |\{ a_1+b_1 = a_2+b_2 ~:~ a_1,a_2 \in A,\, b_1,b_2 \in B \}| \,.
$$
If $A=B$ we simply write $\E^{+} (A)$ instead of $\E^{+} (A,A).$
Clearly,
\begin{equation*}\label{f:energy_convolution}
    \E^{+} (A,B) = \sum_x (A*B) (x)^2 = \sum_x (A \circ B) (x)^2 = \sum_x (A \circ A) (x) (B \circ B) (x)
    \,.
\end{equation*}
Note also that
\begin{equation}\label{f:E_CS}
    \E^{+} (A,B) \le \min \{ |A|^2 |B|, |B|^2 |A|, |A|^{3/2} |B|^{3/2} \} \,,
\end{equation}
and by the Cauchy--Schwarz inequality
\begin{equation}\label{f:E_CS'}
    \E^{+} (A,B) \ge \frac{|A|^2 |B|^2}{|A\pm B|} \,.
\end{equation}
In the same way define the {\it multiplicative energy} of two sets $A,B \subseteq \Gr$
$$
    \E^{\times} (A,B) = |\{ a_1 b_1 = a_2 b_2 ~:~ a_1,a_2 \in A,\, b_1,b_2 \in B \}| \,.
$$
Certainly, multiplicative energy $\E^{\times} (A,B)$ can be expressed in terms of multiplicative convolutions,
similar to (\ref{f:convolutions}).
Let also
$$
    \sigma_k (A) := (A*_k A)(0)=| \{ a_1 + \dots + a_k = 0 ~:~ a_1, \dots, a_k \in A \} | \,.
$$

\bigskip

Now more generally, let
\begin{equation}\label{f:E_k_preliminalies}
    \E_k(A)=\sum_{x\in \Gr} (A\c A)(x)^k  = \sum_{x\in \Gr} |A_x|^k
        = |\{ a_1 - a'_1 = \dots = a_k - a'_k ~:~ a_j, a'_j \in A \}| \,,
\end{equation}
and
\begin{equation}\label{f:E_k_preliminalies_B}
\E_k(A,B)=\sum_{x\in \Gr} (A\c B)^k (x) = |\{ a_1 - b_1 = \dots = a_k - b_k ~:~ a_j \in A,\, b_j \in B \}|
\end{equation}
be the {\it higher energies} of $A$ and $B$, see \cite{SS1}.
Here $A_x= A\cap (A-x)$, $x \in A-A$.
Similarly, we write $\E_k(f,g) = \sum_x (\ov{f} \c g)^k (x)$ for any complex functions $f$, $g$
and even more generally, we consider
$$
    \E_k (f_1,\dots,f_{k}) = \sum_x (\ov{f}_1 \c f_1) (x) \dots (\ov{f}_k \c f_k) (x) \,.
$$

There is a simple connection between $\E_k$ and $\E_2$ energies.
Indeed,
\begin{eqnarray}\label{f:energy-B^k-Delta}
\E_{k+1}(A) = \E(\Delta_k (A),A^{k}) \,,
 \end{eqnarray}
where
$$
    \Delta (A) = \Delta_k (A) := \{ (a,a, \dots, a)\in A^k \}\,.
$$
Identity (\ref{f:energy-B^k-Delta}) allows us to use lower bound (\ref{f:E_CS'}) to estimate higher energies $\E^{+}_k$, $\E^\times_k$.
In particular,
\begin{equation}\label{f:E3_CS}
    |A|^6 \le \E^{+}_3 (A) |A^2 \pm \D(A)|
\end{equation}
and similar for $\E^\times_3 (A)$.


\bigskip

Quantities $\E_k (A,B)$ can be written in terms of generalized convolutions.
Let $f_1,\dots,f_{k+1} : \Gr \to \C$  be functions
and put $F = (f_1,\dots,f_{k+1}) : \Gr^{k+1} \to \C$.
Denote by
$$ \Cf_k(F) (x) = \Cf_{k+1} (f_1,\dots,f_{k+1}) (x_1,\dots, x_{k})$$
the function
$$
    \Cf_k(F) (x) =
    \Cf_{k+1} (f_1,\dots,f_{k+1}) (x_1,\dots, x_{k}) = \sum_z f_1 (z) f_2 (z+x_1) \dots f_{k+1} (z+x_{k}) \,.
$$
Thus, $\Cf_2 (f_1,f_2) (x) = (f_1 \circ f_2) (x)$.
If $f_1=\dots=f_{k+1}=f$ then write
$\Cf_{k+1} (f) (x_1,\dots, x_{k})$ for $\Cf_{k+1} (f,\dots,f_{}) (x_1,\dots, x_{k})$.

\bigskip

Then (see Lemma \ref{l:commutative_C} below) the following holds
$$
    \E_{k+1} (A,B) = \sum_{x_1,\dots,x_k} \Cf_{k+1} (A) (x_1,\dots,x_k) \Cf_{k+1} (B) (x_1,\dots,x_k) \,.
$$

\section{Preliminaries}
\label{sec:preliminaries}

We begin with a
lemma from \cite{s_energy} concerning "commutativity"\, of generalized convolution.

\begin{lemma}
    For any functions $f_i, g_j : \Gr \to \C$ the following holds
$$
    \sum_{x_1,\dots, x_{l-1}} \Cf_l (f_0,\dots,f_{l-1}) (x_1,\dots, x_{l-1})\, \Cf_l (g_0,\dots,g_{l-1}) (x_1,\dots, x_{l-1})
        =
$$
\begin{equation}\label{f:scalar_C}
        =
        \sum_z (f_0 \circ g_0) (z) \dots (f_{l-1} \circ g_{l-1}) (z) \quad \quad  \mbox{\bf (scalar product), }
\end{equation}
moreover
$$
    \sum_{x_1,\dots, x_{l-1}} \Cf_l (f_0) (x_1,\dots, x_{l-1}) \dots  \Cf_l (f_{k-1}) (x_1,\dots, x_{l-1})
        =
$$
\begin{equation}\label{f:gen_C}
        =
            \sum_{y_1,\dots,y_{k-1}} \Cf^l_k (f_0,\dots,f_{k-1}) (y_1,\dots,y_{k-1})
                \quad \quad  \mbox{\bf (multi--scalar product), }
\end{equation}
    and
$$
    \sum_{x_1,\dots, x_{l-1}} \Cf_l (f_0) (x_1,\dots, x_{l-1})\, (\Cf_l (f_1) \circ \dots \circ \Cf_l (f_{k-1})) (x_1,\dots, x_{l-1})
        =
$$
\begin{equation}\label{f:conv_C}
        =
            \sum_{z} (f_0 \circ \dots \circ f_{k-1})^l (z)
                \quad \quad  \mbox{\bf (} \sigma_{k} \quad \mbox{\bf for } \quad \Cf_l \mbox{\bf )}  \,.
\end{equation}
\label{l:commutative_C}
\end{lemma}

The next lemma shows that "higher sum sets"\, can be expressed in terms of ordinary sums, see e.g. \cite{SS1}.

\begin{lemma}
    Let $A\subseteq \Gr$ be a set.
    Then
\begin{equation}\label{f:D(A)}
    |A^2 \pm \D(A)| = \sum_{x\in A-A} |A \pm A_x| \,.
\end{equation}
\label{l:D(A)}
\end{lemma}

The main objects of the paper are two quantities $D^{+} (A)$, $D^\times (A)$.
Let us recall the
definitions.

\begin{definition}
A finite set $A \subset \R$ is said to be of {\it additive Szemer\'{e}di--Trotter type}
with a parameter $D^{+} (A) >0$ if the inequality
\begin{equation}\label{f:SzT-type}
\bigl|\bigl\{ s\in A-B ~\mid~ |A\cap (B+s)| \ge \tau \big\}\big|
    \le
\frac{D^{+} (A) |A| |B|^2}{\tau^{3}}\,,
\end{equation}
holds
 for every finite set $B\subset \R$ and every real number $\tau \ge 1$.
\label{def:SzT-type}
\end{definition}

 The quantity $D^{+} (A)$ can be considered as the infimum of numbers $D$ such that (\ref{f:SzT-type})
 takes place
 for any $B$ and $\tau \ge 1$
 but, of course, the definition is applicable just for sets $A$ with  small quantity $D^+ (A)$.
 It is easy to see
 that $D^{+} (A) \le |A|$, so
 $|A|$ can be considered as a trivial upper bound for the quantity.
 Note also that $D^{+} (A) \ge 1$ (just take $B$ equals any one--element set and substitute $\tau =1$ into formula (\ref{f:SzT-type})).

Any SzT--type set has small number of solutions of a wide class of linear equations, see e.g. Corollary 8 from \cite{KS1} (where nevertheless another quantity $D^+ (A)$ was used)
and
Lemmas  7, 8 from \cite{s_sumsets}, say.

\begin{lemma}
    Let  $A_1,A_2 \subset \R$ be any finite sets.
    Then
$$
    \E^{+} (A_1,A_2) \ll (D^+ (A_1))^{1/2} |A_1| |A_2|^{3/2} \,,
$$
    and
$$
    \E^{+}_3 (A_1,A_2) \ll D^+ (A_1) |A_1| |A_2|^{2} \cdot \log (\min\{ |A_1|, |A_2|\}) \,.
$$
\label{l:gen_sigma}
\end{lemma}

Similarly, consider a dual characteristic of a set of real numbers.

\begin{definition}
A finite set $A \subset \R$ is said to be of {\it multiplicative Szemer\'{e}di--Trotter type}
with a parameter $D^{\times} (A) >0$ if the inequality
\begin{equation}\label{f:SzT-type_m}
\bigl|\bigl\{ s\in A/B ~\mid~ |A\cap sB| \ge \tau \big\}\big|
    \le
\frac{D^{\times} (A) |A| |B|^2}{\tau^{3}}\,,
\end{equation}
holds
 for every finite set $B\subset \R$ and every real number $\tau \ge 1$.
\label{def:SzT-type_m}
\end{definition}

Of course a multiplicative  analog of Lemma \ref{l:gen_sigma} takes place.

\begin{lemma}
    Let  $A_1,A_2 \subset \R$ be any finite sets.
    Then
$$
    \E^{\times} (A_1,A_2) \ll (D^\times (A_1))^{1/2} |A_1| |A_2|^{3/2} \,,
$$
    and
$$
    \E^{\times}_3 (A_1,A_2) \ll D^\times (A_1) |A_1| |A_2|^{2} \cdot \log (\min\{ |A_1|, |A_2|\}) \,.
$$
\label{l:gen_sigma_m}
\end{lemma}

\section{Norms $\E_k$}
\label{sec:norms}


For any function $f : \Gr \to \C$ and an arbitrary integer $k\ge 1$ put
\begin{equation}\label{f:E_k_norm}
    \| f \|_{E_k} := (\E_k (f))^{1/2k} = (\sum_{x} (\ov{f}\c f) (x)^{k} )^{1/2k}
    \,.
\end{equation}
By formula (\ref{f:F_svertka}), we get
\begin{equation}\label{f:E_k_norm'}
    \| f \|^{2k}_{E_k} = N^{-(k-1)} \sum_{x_1+\dots+x_k = 0} |\FF{f} (x_1)|^2 \dots |\FF{f} (x_k)|^2
\end{equation}
and hence the expression is nonnegative.
Another way is to think about $\| f \|_{E_k}$ is to note that by formula (\ref{f:scalar_C}) of Lemma \ref{l:commutative_C}, we have
\begin{equation}\label{f:E_k_norm''}
    \| f \|^{2k}_{E_k} = \sum_{x_1,\dots,x_{k-1}} |\Cf_{k} (f) (x_1,\dots,x_{k-1})|^2 \,.
\end{equation}
Note that there are nonzero functions $f$ with $\| f \|_{E_k} = 0$, e.g. $\Gr = \F_p$, $p$ is a prime number, $k< p$
and $f(x) = e^{2\pi ix/p}$.
If we restrict ourselves to consider just real functions then again it is possible to find  nonzero functions $f$ with $\| f \|_{E_k} = 0$.

\begin{example}
    Let $\Gr = \f_2^n$ and $f(x) = f(x_1,\dots,x_n) = (-1)^{x_1+\dots+x_n}$.
    Then $\FF{f} (r) = 0$ for any $r\neq (1,\dots,1)$.
    Thus by formula (\ref{f:E_k_norm'}), we have $\| f \|_{E_k} = 0$ for all odd $k$.
\end{example}

If $k \ge 2$ is even then the last situation is not possible.

\begin{lemma}
    Let $f : \Gr \to \R$ be a function and $k\ge 2$ be an even number.
    Then $\| f\|_{E_k} = 0$ iff $f\equiv 0$.
\label{l:zero_E_k}
\end{lemma}
\begin{proof}
    We give even three proofs.
    The first one uses Fourier analysis and another two do not.
    Applying  formula (\ref{f:E_k_norm'}) we see that
\begin{equation}\label{tmp:20.04.2016_3}
    |\FF{f} (x_1)|^2 \dots |\FF{f} (x_k)|^2 = 0
\end{equation}
    for all $x_1,\dots, x_k \in \Gr$ such that $x_1 + \dots +x_k =0$.
    By the assumption $f$ is a real function, thus $\FF{f} (-x) = \ov{\FF{f} (x)}$.
    Using
    the fact and substitute  variables
    $x_1 =x$, $x_2=-x$, $x_3=x$, $x_4 = -x, \dots$, $x_{k-1}=x$, $x_k=x$ into formula (\ref{tmp:20.04.2016_3}), we obtain
    $
        |\FF{f} (x)|^{2k} = 0
    $
    for every $x$ and hence $f\equiv 0$.

    In our second proof, we use
     identity
     (\ref{f:E_k_norm''})
    and see that
$$
    \Cf_{k} (f) (x_1,\dots,x_{k-1}) = \sum_{z} f(z) f(z+x_1) \dots f(z+x_{k-1}) = 0
$$
    for any $x_1,\dots,x_{k-1} \in \Gr$.
    Put $x_1 = \dots = x_{k-1} =0$ and using the assumption that $k$ is even as well as $f$ is a real function, we obtain $\sum_z f^{k} (x) = \sum_z |f(z)|^k = 0$ which implies $f\equiv 0$.

    Finally, applying our conditions we see from definition (\ref{f:E_k_norm}) that $\| f\|_{E_k} \ge \| f \|_2$
    (just substitute $x=0$ in  (\ref{f:E_k_norm})).
This completes the proof.
$\hfill\Box$
\end{proof}

\begin{lemma}
    For any $k\ge 1$ and arbitrary functions $f_1,f'_1, \dots,f_{k}, f'_k : \Gr \to \C$ the following holds
\begin{equation}\label{f:E_k_prod}
    \left| \sum_x (f_1 \c f'_1) (x) \dots (f_{k} \c f'_{k}) (x) \right|
        \le
            \prod_{j=1}^{k} \| |f_j| \|_{E_k}  \| |f'_j| \|_{E_k} \,.
\end{equation}
    If $k\ge 2$ is even and all functions are real then one can remove modulus from formula (\ref{f:E_k_prod}).
\label{l:E_k_prod}
\end{lemma}
\begin{proof}
By formula (\ref{f:scalar_C}) of Lemma \ref{l:commutative_C}, we have
$$
    \sigma:= \sum_x (f_1 \c f'_1) (x) \dots (f_{k} \c f'_{k}) (x)
        =
            \sum_{x_1,\dots,x_k}
                \Cf_{k+1} (f_1,\dots,f_{k}) (x_1,\dots,x_k) \Cf_{k+1} (f'_1,\dots,f'_{k}) (x_1,\dots,x_k) \,.
$$
Using the Cauchy--Schwarz inequality and formula (\ref{f:scalar_C}) of Lemma \ref{l:commutative_C} again, we obtain
$$
    |\sigma|^2
        \le
            \sum_{x_1,\dots,x_k} |\Cf_{k+1} (f_1,\dots,f_{k}) (x_1,\dots,x_k)|^2
                \cdot
                    \sum_{x_1,\dots,x_k} |\Cf_{k+1} (f'_1,\dots,f'_{k}) (x_1,\dots,x_k)|^2
                        =
$$
$$
    = \sum_x (\ov{f_1} \c f_1) (x) \dots (\ov{f_{k}} \c f_{k}) (x)
        \cdot
        \sum_x (\ov{f'_1} \c f'_1) (x) \dots (\ov{f'_{k}} \c f'_{k}) (x) \,.
$$
By the H\"{o}lder inequality it is sufficient to prove that
$$
    |\sigma| = \sum_x |(\ov{f} \c f) (x)|^k
        \le
            \sum_x (|f| \c |f|) (x)^k
        =
             \| |f| \|_{E_k}^{2k}
$$
for any function $f : \Gr \to \C$.
If $k$ is even and $f$ is a real function then we need to check that
$$
    \sigma = \sum_x (f \c f) (x)^k = \| f \|^{2k}_{E_k} \,.
$$
The last two formulas
coincide with
the definition of the norm $E_k$.
This completes the proof.
$\hfill\Box$
\end{proof}

\bigskip

\begin{example}
    Let $\Gr = \f_2^n$, $f_1 (x) = f_{\v{\la}'} (x) = f_1 (x_1,\dots,x_n) = (-1)^{\la_1 x_1 + \dots + \la_n x_n}$,
    $\v{\la}' = (\la_1,\dots,\la_n)$ and similar $f_2 (x) = f_{\v{\la}''} (x)$, $f_3 (x) = f_{\v{\la}'''} (x)$.
    Take $\v{\la}'$, $\v{\la}''$, $\v{\la}'''$ be three nonzero vectors such that
    $\v{\la}' + \v{\la}'' + \v{\la}''' = 0$.
    Then by simple calculations, we get
$$
    \sum_{x} (f_1 \c f_1) (x) (f_2 \c f_2) (x) (f_3 \c f_3) (x) = 2^{4n}
$$
    but for any $j=1,2,3$ one has $\sum_{x} (f_j \c f_j)^3 (x) = 0$.
    Thus (\ref{f:E_k_prod}) does not hold without modula in the case of odd $k$ and $f_j$ are real functions.
\end{example}

\bigskip

We need a
combinatorial lemma.
Let $l$ be a positive integer and let $\Omega_l = \{ 0,1 \}^l$.
For any $\eps \in \Omega_l$ put $\wt(\eps)$ equals the number of ones in $\eps$.
Finally, given  numbers $k_1,\dots, k_s$ such that $k_1+\dots+k_s = k$ write $\binom{k}{k_1, \dots, k_s}$ for
$\frac{k!}{k_1! \dots k_s!}$.

\begin{lemma}
    Let $n,k,l$ be positive integers, $n\le lk$.
    Then
\begin{equation}\label{f:comb1}
    \sum_{\sum_{\eps \in \Omega_l} n_\eps =k,\,\, \sum_{\eps \in \Omega_l} \wt(\eps) n_\eps = n}\,
        \frac{k!}{\prod_{\eps \in \Omega_l} n_\eps!} = \frac{(lk)!}{n!(lk-n)!} \,.
\end{equation}
    In particular,
\begin{equation}\label{f:comb2}
    \sum_{n_1+n_2+n_3+n_4=k,\, 2n_1 + n_2 + n_3 = n}\, \frac{k!}{n_1! n_2! n_3! n_4!} = \frac{(2k)!}{n!(2k-n)!} \,.
\end{equation}
\label{l:comb}
\end{lemma}
\begin{proof}
    One has
$$
    \sum_{n} \binom{lk}{n} x^n = (1+x)^{lk} = \left( \sum_{i=0}^l \binom{l}{i} x^i \right)^k
        = \sum_{m_0+m_1+\dots+m_l = k}\, x^{\sum_{j=1}^l j m_j} \binom{k}{m_0, m_1, \dots, m_l}
            \prod_{r=0}^l \binom{l}{r}^{m_r} \,.
$$
It follows that
$$
    \binom{lk}{n} =  \sum_{m_0+m_1+\dots+m_l = k,\,\, \sum_{j=1}^l j m_j = n}\, \binom{k}{m_0, m_1, \dots, m_l}
            \prod_{r=0}^l \binom{l}{r}^{m_r}
            =
$$
$$
    =
        \sum_{m_0+m_1+\dots+m_l = k,\,\, \sum_{j=1}^l j m_j = n}\, \binom{k}{m_0, m_1, \dots, m_l}
            \prod_{r=0}^l\, \sum_{s_1+\dots + s_{\binom{l}{r}} = m_r} \binom{m_r}{s_1, \dots, s_{\binom{l}{r}}} \,.
$$
    Fix $r \in 0,\dots,l$ and redenote $\binom{l}{r}$ variables $s_{\binom{l}{r}}$ by $n_\eps$, $\eps \in \Omega_l$ such that $\wt (\eps) = r$.
    Hence we have redenoted all $2^l$ variables $s_{\binom{l}{r}}$, $r=0,\dots,l$ as $n_\eps$, $\eps \in \Omega_l$.
    Note that $\sum_{\eps \in \Omega_l} n_\eps = m_0+m_1+\dots+m_l = k$.
    Further, by
    our choice of enumeration of $n_\eps$, we get
    $\sum_{\eps \in \Omega_l} \wt(\eps) n_\eps = \sum_{j=1}^l j m_j = n$.
    Thus, we obtain
$$
    \binom{lk}{n}
        =
            \sum_{\sum_{\eps \in \Omega_l} n_\eps =k,\,\, \sum_{\eps \in \Omega_l} \wt(\eps) n_\eps = n}\,
            \frac{k!}{\prod_{\eps \in \Omega_l} n_\eps!}
$$
as required.
$\hfill\Box$
\end{proof}

\bigskip

Using the lemmas above we are ready to prove the main result of the section.

\begin{proposition}
    Let $k\ge 2$ be an integer.
    Then for any pair of functions $f,g : \Gr \to \C$ the following holds
    $$\| f + g\|_{E_k} \le \| |f| \|_{E_k} + \| |g| \|_{E_k} \,.$$
    If
    we consider just real functions
    and $k$ is even then
    $\| \cdot \|_{E_k}$ is a norm.
\label{p:E_k-norm}
\end{proposition}
\begin{proof}
We have
$$
    \sigma:= \| f + g\|^{2k}_{E_k} = \sum_{x} ((\ov{f}\c f)(x) + (\ov{f}\c g)(x) + (\ov{g}\c f) (x) + (\ov{g}\c g)(x))^k
        =
$$
$$
        =
            \sum_{n_1+n_2+n_3+n_4=k} \binom{k}{n_1,n_2,n_3,n_4}
                \sum_x (\ov{f}\c f)^{n_1} (x) (\ov{f}\c g)^{n_2}(x) (\ov{g}\c f)^{n_3} (x) (\ov{g}\c g)^{n_4} (x) \,.
$$
Using Lemma \ref{l:E_k_prod} and Lemma \ref{l:comb} in the case $l=2$, we get
$$
    \sigma \le \sum_{n_1+n_2+n_3+n_4=k} \binom{k}{n_1,n_2,n_3,n_4}
        \| |f| \|^{2n_1+n_2+n_3}_{E_k} \| |g| \|^{n_2+n_3+2n_4}_{E_k}
        =
$$
$$
        =
            \sum_{n+m=2k} \frac{(2k)!}{n! m!} \| |f| \|^{n}_{E_k} \| |g| \|^{m}_{E_k}
                = (\| |f| \|_{E_k} + \| |g| \|_{E_k})^{2k}
$$
as required.

If $f,g$ are real functions and $k$ is even then we apply the second part of Lemma \ref{l:E_k_prod},
and obtain $\| f + g\|_{E_k} \le \| f \|_{E_k} + \| g \|_{E_k} \,.$
Also Lemma \ref{l:zero_E_k} says that $\| f \|_{E_k} = 0$ iff $f\equiv 0$ in the case.
This completes the proof.
$\hfill\Box$
\end{proof}

\section{The proof of Theorem \ref{t:BW_D_introduction}}
\label{sec:proof}

In the next two sections we prove Theorems \ref{t:R_E_3_introduction}, \ref{t:BW_D_introduction}.
We begin with the stronger Theorem \ref{t:BW_D_introduction} and after that use similar arguments, combining with the results of section \ref{sec:norms} to get Theorem \ref{t:R_E_3_introduction}.

\bigskip

First of all express quantities $D^{+} (A)$, $D^\times (A)$ in terms of the energies $\E_3^{+} (A,B)$, $\E_3^{\times} (A,B)$. Consider the case of $D^{+} (A)$ the second variant is similar.
For any finite set $A\subset \R$ put
\begin{equation}\label{def:q}
    q^{+} (A) := \max_{B\neq \emptyset} \frac{\E^{+}_3(A,B)}{|A| |B|^2} \,.
\end{equation}
It is easy to see that the maximum in (\ref{def:q}) is attained.
Indeed, shifting we can suppose that $0\in B$, further the size of $B$ is bounded in terms of $|A|$ and by Lemma \ref{l:commutative_C} one has
\begin{equation}\label{f:E_3_fL}
    \E^{+}_3(A,B) = \sum_{(x,y) \in A^2 - \D(A)} \Cf_3 (B) (x,y) \Cf_3 (A) (x,y) \,.
\end{equation}
Whence by induction we show that $B\subseteq k(A-A)$, where $k$ is
bounded in terms of $|A|$.
Thus the maximum in (\ref{def:q}) is attained.

Let us
give
some simple  bounds for $q^{+} (A)$.
In view of (\ref{f:E_CS}), (\ref{f:E_k_preliminalies}), we have, clearly,
\begin{equation}\label{f:q_lower}
    1 \le \frac{\E^{+}_3(A)}{|A|^3} \le q^{+} (A) \le |A| \,.
\end{equation}
More precisely,
by
formulas (\ref{f:E_CS}), (\ref{f:E_3_fL}) and the Cauchy--Schwarz inequality, one has
\begin{equation}\label{f:q_E_3_trivial}
    q^{+} (A) \le \frac{(\E^{+}_3 (A))^{1/2}}{|A|} \le |A| \,.
\end{equation}

Now we are ready to show that $D^{+} (A)$ is proportional to $q^{+} (A)$ up to logarithms.

\begin{lemma}
    Let $A\subset \R$ be a finite set.
    Then
\begin{equation}\label{f:D_q}
    q^{+} (A) \lesssim D^{+} (A) \le q^{+} (A) \,.
\end{equation}
\label{l:D_q}
\end{lemma}
\begin{proof}
    The lower bound in (\ref{f:D_q}) immediately  follows from Lemma \ref{l:gen_sigma}.
    Now suppose that for some $B \subset \R$ and $\tau \ge 1$ the following holds
    $D^{+} (A) |A| |B|^2 = \tau^3 |S_\tau (A,B)|$,
    where
$$
    S_\tau (A,B) := \bigl\{ s\in A-B ~:~ |A\cap (B+s)| \ge \tau \big\} \,.
$$
    Hence
$$
    \E_3^{+} (A,B) = \sum_{s} (A\c B)^3 (s) \ge \sum_{s \in S_\tau (A,B)} (A\c B)^3 (s)
        \ge \tau^3 |S_\tau (A,B)| = |A| |B|^2 D^{+} (A) \,.
$$
    The last formula implies that $q^{+}(A) \ge D^{+} (A)$.
This completes the proof.
$\hfill\Box$
\end{proof}

\bigskip

Lemma above, combining with Lemmas \ref{l:gen_sigma}, \ref{l:gen_sigma_m} and inequality (\ref{f:q_E_3_trivial}), implies that there is a close connection between quantities $D^{+} (A)$, $D^\times (A)$ and $\E^{+}_3 (A)$, $\E^{\times}_3 (A)$, correspondingly.

\begin{corollary}
    Let $A \subset \R$ be a finite set. Then
$$
    |A|^2 (D^{+} (A))^2 \le \E^{+}_3 (A) \lesssim D^{+} (A) |A|^3 \,.
$$
    The same holds for $D^\times (A)$.
\end{corollary}

\bigskip

Secondly, let us note that the quantities  $D^{+}$, $D^\times$ have some kind of "subadditive"\, property.
Actually, our results hold in a general abelian group $\Gr$.

\begin{lemma}
    Let $A_1,\dots, A_k \subset \R$ be finite sets.
    Then
$$
    D(A_1 \cup \dots \cup A_k) \lesssim k^3 \cdot \frac{1}{|A|} \sum_{j=1}^k D(A_j) |A_j| \,.
$$
    where $D=D^+$ or $D^\times$.
\label{l:subadd_D}
\end{lemma}
\begin{proof}
Let us consider the case of $D^{\times}$, the situation with $D^{+}$ is similar.
We give even two proofs.
Put $A=A_1 \cup \dots \cup A_k$.
Take any set $G\subset \R$ and a real number $\tau \ge 1$.
We need to estimate the size of the set
$$
    S_\tau = S_\tau (A,G) := \bigl\{ s\in A/G ~:~ |A\cap sG| \ge \tau \big\} \,.
$$
For any $s\in S_\tau$ one has
\begin{equation}\label{tmp:28.04.2016_1}
    \tau \le |A\cap sG| \le \sum_{j=1}^k |A_j \cap sG| \lesssim \sum_{j\in \Omega_\D (s)} |A_j \cap sG|
        \le
            2 \D(s) |\Omega_\D (s)| \,,
\end{equation}
    where we write $\D(s)$
    for
    some number of the form $2^j$, $j\ge 0$ such that
    $$
        \Omega_\D (s) = \{ j ~:~ \D(s) < |A_j \cap sG| \le 2 \D(s) \} \subseteq [k]
    $$
    and
    $\tau k^{-1} \lesssim \D(s)$.
    The existence of the number $\D(s)$ and the set $\Omega_\D (s)$ follows from the pigeonhole principle.
    Using the pigeonhole principle once more time, we find a set $S'_\tau \subseteq S_\tau$ such that
    $|S'_\tau| \gtrsim |S_\tau|$,
    further, the number $\D$ with
    $$
        \Omega_\D (s) = \{ j ~:~ \D < |A_j \cap sG| \le 2 \D \} \subseteq [k] \,,
    $$
\begin{equation}\label{tmp:23.03.2016_1}
    \tau k^{-1} \lesssim \D
\end{equation}
    and such that  for any $s\in S'_\tau$, we have an analog of (\ref{tmp:28.04.2016_1}), namely,
\begin{equation*}\label{tmp:28.04.2016_1}
    \tau \le |A\cap sG|
            \lesssim
            \D |\Omega_\D (s)| \,.
\end{equation*}
    By
    our
    construction, we have $S'_\tau \subseteq \bigcup_{j=1}^k S_\D (A_j,G)$.
    In view of (\ref{tmp:23.03.2016_1}),
    we obtain
$$
    |S_\tau| \lesssim |S'_\tau| \le \sum_{j=1}^k |S_\D (A_j, G)| \le \frac{|G|^2}{\D^3} \sum_{j=1}^k D(A_j) |A_j|
        \lesssim
            \frac{|G|^2 k^3}{\tau^3} \sum_{j=1}^k D(A_j) |A_j|
                =
$$
$$
                =
                    \frac{|G|^2 |A| k^3}{\tau^3} \cdot \frac{1}{|A|} \sum_{j=1}^k D(A_j) |A_j| \,.
$$
By the definition it means that $D(A) \lesssim \frac{k^3}{|A|} \sum_{j=1}^k D(A_j) |A_j|$ as required.

Now let us give another proof via quantity $q$.
Applying the H\"{o}lder inequality and Lemmas \ref{l:gen_sigma}, \ref{l:gen_sigma_m}, we have
$$
    \E_3 (\bigcup_{i=1}^k A_i,B) \le \sum_x (\sum_{i=1}^k A_i \c B)^3 (x)
        = \sum_x \left( \sum_{i=1}^k (A_i \c B) (x) \right)^3
            =
$$
$$
            =
                \sum_{i_1,i_2,i_3=1}^k \sum_x (A_{i_1} \c B) (x) (A_{i_2} \c B) (x) (A_{i_3} \c B) (x)
                    \le
                        \left( \sum_{i=1}^k \left( \sum_x (A_i \c B)^3 (x) \right)^{1/3} \right)^3
                            \le
$$
$$
                            \le
                                k^3 \sum_{i=1}^k \sum_x (A_i \c B)^3 (x)
                                    \lesssim
                                        k^3 |A| |B|^2 \cdot \frac{1}{|A|} \sum_{j=1}^k D(A_j) |A_j| \,.
$$
Thus, by Lemma \ref{l:D_q}, we obtain $D(A) \le q(A) \lesssim \frac{k^3}{|A|} \sum_{j=1}^k D(A_j) |A_j|$,
where $q$ equals $q^{+}$ or $q^\times$, correspondingly to $D$.
This completes the proof.
$\hfill\Box$
\end{proof}

\bigskip

In \cite{KS2} we considered two another characteristics of
$A$.
Put
$$
    \Sym^\times_{t} (Q,R) = \{ x ~:~ |Q \cap xR^{-1}| \ge t \} \,,
$$
and
\begin{equation}\label{f:d_r}
    d^{+} (A) = \min_{t>0}\, \min_{\emptyset \neq Q,R \subset \R \setminus \{0\}
                    } \,
        \frac{|Q|^2 |R|^2}{|A| t^3} \,,
\end{equation}
    where the second minimum in (\ref{f:d_r}) is taken over any $Q,R$ such that $A\subseteq \Sym^\times_{t} (Q,R)$
    and $\max\{ |Q|,|R| \} \ge |A|$.

Similarly, for any sets $Q,R$ and a real number $t>0$ put
$$
    \Sym^+_{t} (Q,R) := \{ x ~:~ |Q\cap (x-R)| \ge t \}
$$
and consider the following quantity
\begin{equation}\label{f:d_r+}
    d^\times (A) := \min_{t>0}\, \min_{\emptyset \neq Q,R \subset \R \setminus \{0\}
                    } \,
        \frac{|Q|^2 |R|^2}{|A| t^3} \,,
\end{equation}
    where the second minimum in (\ref{f:d_r+}) is taken over any $Q,R$ such that
    $A\subseteq \Sym^+_{t} (Q,R)$
    and $\max\{ |Q|,|R| \} \ge |A|$.
    It is easy to see \cite{KS2} that $1 \le d^{+} (A), d^{\times} (A) \le |A|$.
    In \cite{KS2} the following result was proved.

\begin{lemma}
    Let $A\subset \R$ be a finite set.
    Then $A$ is of additive Szemer\'{e}di--Trotter type with $O(d^{+} (A))$
    and $A$ is of multiplicative Szemer\'{e}di--Trotter type with $O(d^{\times} (A))$.
\label{l:d_r}
\end{lemma}

\bigskip

Now we can formulate a new result, which implies Theorem \ref{t:BW'} with $\d= 1/5$
if one combines Theorem \ref{t:BW_D} below  with Lemmas \ref{l:gen_sigma}, \ref{l:gen_sigma_m}.

\begin{theorem}
    Let $A\subset \R$ be a finite set and $\d = 2/5$.
    Then there are two disjoint subsets $B$ and $C$ of $A$ such that $A = B\sqcup C$ and
$$
    \max\{ D^{+} (B), D^{\times} (C)\} \lesssim  |A|^{1-\d} \,.
$$
\label{t:BW_D}
\end{theorem}
\begin{proof}
    Let $1\le M \le |A|$ be a parameter which we will choose later.
    Our arguments is a sort of an algorithm.
     We construct a decreasing sequence of sets $C_1=A \supseteq C_2 \supseteq \dots \supseteq C_k$ and an increasing sequence of sets $B_0 = \emptyset \subseteq B_1 \subseteq \dots \subseteq B_{k-1} \subseteq A$ such that
    for any $j=1,2,\dots, k$ the sets $C_j$ and $B_{j-1}$  are disjoint and moreover $A = C_j \sqcup B_{j-1}$.
    If at some step $j$ we have $D^{\times} (C_j) \le |A| / M$ then we stop our algorithm putting
    $C=C_j$, $B = B_{j-1}$, and $k=j-1$.
    Consider the opposite situation, that is $D^{\times} (C_j) > |A| / M$.
    Put $C' = C_j$.
    By the definition there exists a number $\tau \ge 1$ and a finite set $G = G_j\subset \R$ such that the set
$$
    S_\tau = S_\tau (G_j,C') := \bigl\{ s\in C'/G ~:~ |C'\cap sG| \ge \tau \big\}
$$
    has size at least $\frac{|C'|^2 |G|^2}{M \tau^3}$.
    We have
$$
    \tau |S_\tau| \le \sum_{s\in S_\tau} |C' \cap sG| = \sum_{a\in C'} |S_\tau \cap a G^{-1}| \,.
$$
   By the pigeonholing principle there is  a set $A'\subseteq C'$ and a number $q$ such that
$|A'| q \sim \tau |S_\tau|$ and $q< |S_\tau \cap a G^{-1}| \le 2 q$ for any $a\in A'$.
In other words, $A' \subseteq \Sym^\times_q (S_\tau,G)$.
    Applying Lemma \ref{l:d_r} with $P=S_\tau$, $Q=G$, we get
\begin{equation}\label{tmp:22.03.2016_1}
    D^{+} (A') \ll d^{+} (A') \le \frac{|S_\tau|^2 |G|^2}{q^3 |A'|} \,.
\end{equation}
    Further, we know that $|A'| q \sim \tau |S_\tau|$ and $|S_\tau| > \frac{|C'|^2 |G|^2}{M \tau^3}$.
    Combining these inequality with bound (\ref{tmp:22.03.2016_1}), we obtain
\begin{equation}\label{f:A'C'_est}
     D^{+} (A') \ll d^{+} (A') \lesssim \frac{|A'|^2 |G|^2}{\tau^3 |S_\tau|} < \frac{|A'|^2 M}{|C'|^2} \,.
\end{equation}
    Since $1\le d^{+} (A')$, we have $|A'| \gtrsim |C'| M^{-1/2}$.
    Trivially, $A' \subseteq C'$ and hence $D^{+} (A') \lesssim M$.

    After that we put $D_j = A'$, $C_{j+1} = C_j \setminus D_j$, $B_j = B_{j-1} \sqcup D_j$ and repeat the procedure.
    Clearly, at step $k$ one has $B_k = \bigsqcup_{j=1}^k D_j$ and because of $|D_j| \gtrsim |C_j| M^{-1/2}$, we have after some calculations that $k\lesssim M^{1/2}$, so $k$ is finite.
    It remains to estimate $D^{+} (B_k) = D^{+} (D_1 \cup \dots \cup D_k)$.
    Put $U_j = \{ i \in [k] ~:~ 2^{j-1} \le D^{+} (D_i) \le 2^j \}$, $j\in [t]$ and $k_j = |U_j|$.
    Since for any $i$ one has  $D^{+} (D_i) \lesssim M$ it follows that the number of sets $U_j$ is $t\lesssim 1$.
    Let also $B^{(j)}_k = \bigcup_{i\in U_j} D_i$.
    Applying Lemma \ref{l:subadd_D}, we get
    \begin{equation}\label{tmp:27.04.2016_1}
        D^{+} (B_k) \lesssim \frac{1}{|B_k|} \sum_{j=1}^t D^{+} (B^{(j)}_k) |B^{(j)}_k| \,.
    \end{equation}
    Now fixing $j\in [t]$, we see that $k_j \lesssim M^{1/2} 2^{-j/2}$.
    Using Lemma \ref{l:subadd_D} once more time, we obtain
    $$
        D^{+} (B^{(j)}_k) \lesssim  \frac{k^3_j}{|B^{(j)}_k|}\, \sum_{i\in U_j} D^{+} (D_i) |D_i|
            \lesssim
                M^{3/2} 2^{-j/2} \frac{1}{|B^{(j)}_k|}\, \sum_{i\in U_j} |D_i|
                    =
                         M^{3/2} 2^{-j/2} \,.
    $$
    Substituting the last bound into (\ref{tmp:27.04.2016_1}), we find
    $$
        D^{+} (B_k) \lesssim \frac{M^{3/2}}{|B_k|} \sum_{j=1}^t 2^{-j/2} |B^{(i)}_k|
            \ll
                M^{3/2} \,.
    $$
    Optimizing over $M$, that is solving the equation $M^{3/2} = |A|/M$ and choosing $M=|A|^{2/5}$, we obtain the result.
    This completes the proof.
$\hfill\Box$
\end{proof}

\bigskip

As for lower bounds in Theorems  \ref{t:R_E_3_introduction}, \ref{t:BW_D_introduction}, we use small modification of the construction from \cite{BW}.
A counterexample is so--called $(H+\Lambda)$--sets, see \cite{s_energy}.

\begin{theorem}
    For any positive integer $N$ there exists a set $A\subseteq \N$, $|A| =N$
    such that for an arbitrary $B\subseteq A$ with $|B| \ge |A|/2$ one has
    $\E^{+}_3 (B), \E^{\times}_3 (B) \gg |A|^{13/4}$
    and
    $D^{+} (B), D^{\times} (B) \gtrsim |A|^{1/4}$.
\label{t:construction_BW}
\end{theorem}
\begin{proof}
    Take an integer parameter $1\le K \ll N$, which we will choose later,
    $t=\lceil N/K \rceil$ and put $G = \{ 2^i\}_{i=1}^K$, $P = \{3=p_1 < p_2 < \dots < p_t\}$ be $t$ consecutive odd primes.
    Finally, put $A=PG$, $|A| = tK = N + \theta K$, where $|\theta| \le 1$.
    Thus,
    redefining $N$ if needed
    one can think that $|A| =N$.

    Consider any $B\subseteq A$ such that $|B| \ge |A|/2$.
    For any $j\in [K]$ put $B_j = B \cap (P \cdot 2^j)$.
    Clearly, by the theorem on the density of the primes and estimate (\ref{f:E_CS'}) or (\ref{f:E3_CS}), we get
    $$
        \E_{3}^+ (B_j) \ge \frac{|B_j|^6}{|P+P|^2} \gtrsim \frac{|B_j|^6}{|P|^2} \,.
    $$
    Thus using the Cauchy--Schwarz inequality, we obtain
\begin{equation}\label{tmp:20.04.2016_2}
    \E^{+}_3 (B) \ge \sum_{j=1}^K \E^{+}_3 (B_j) \gtrsim |P|^{-2} \sum_{j=1}^K |B_j|^6
        \ge
            \frac{|B|^6}{|P|^2 K^5} \gg \frac{N^4}{K^3} \,.
\end{equation}

    Now let us
    calculate
    $\E_3^{\times} (B)$.
    By formula (\ref{f:E3_CS}), we have
\begin{equation}\label{tmp:20.04.2016_-1}
    |A|^6 \ll |B|^6 \le \E^\times_3 (B) |B^2 \cdot \D(B)| \le \E^\times_3 (B) |A^2 \cdot \D(A)|
\end{equation}
    and thus it is sufficient to estimate the size of the set $A^2 \cdot \D(A)$.
    Put $A_x= A\cap xA$.
    Applying
    formula (\ref{f:D(A)}) of Lemma \ref{l:D(A)} in its multiplicative form, we obtain
$$
    |A^2 \cdot \D(A)| = \sum_{x\in A/A} |AA_x| = \sum_{x\in G/G} |A A_x| + \sum_{x\in (A/A) \setminus (G/G)} |A A_x|
        \le
$$
\begin{equation}\label{tmp:20.04.2016_1}
        \le
            |G/G| |GGPP| + \sum_{x\in (A/A) \setminus (G/G)} |A A_x|
                \ll
                    N^2 + \sigma \,.
\end{equation}
Let us prove that $\sigma \ll N^3 / K$.
Put $x\in (A/A) \setminus (G/G)$.
In other words $x=g_1/g_2 \cdot p_1/p_2$ and $p_1\neq p_2$.
Now taking $a\in A_x$, we have
$$
    a= p' g' = g_1 p_1 p'' g'' / g_2 p_2
$$
or
$$
    p' p_2 \cdot g' g_2 = p'' p_1 \cdot g'' g_1 \,,
$$
where $p',p'' \in P$, $g',g'' \in G$.
Thus $p' = p_1$, $p''= p_2$ and $g' g_2 = g'' g_1$.
Hence $A_x = p_1 g_1 g^{-1}_2 \cdot G$ and $|A A_x| \le |GG P| \le 2 N$.
It follows that
$$
    \sigma \le 2|A/A|N \ll  N |G/G \cdot P/P| \ll N^3 / K \,.
$$
Returning to (\ref{tmp:20.04.2016_-1}), (\ref{tmp:20.04.2016_1}) and recalling that $K\ll N$, we get
$$
    N^6 \ll \E_3^{\times} (B) \cdot \frac{N^3}{K} \,.
$$
In view of (\ref{tmp:20.04.2016_2}) the optimal choice of $K$ is $K\sim N^{1/4}$.

Finally, inequalities  $D^{+} (B), D^{\times} (B) \gtrsim |A|^{1/4}$ immediately
follows from the obtained lower bounds for $\E^{+}_3 (B), \E^{\times}_3 (B)$, formula (\ref{f:q_lower}) and Lemma \ref{l:D_q}.
This completes the proof.
$\hfill\Box$
\end{proof}

\bigskip

Combining Theorem \ref{t:construction_BW} with Theorem \ref{t:BW_D}, we obtain Theorem \ref{t:BW_D_introduction} as well as
Theorem \ref{t:R_E_3_introduction}.
It is easy to see that estimate (\ref{f:BW_1}) implies bound (\ref{f:BW_2}) of Theorem \ref{t:BW} via the Cauchy--Schwarz inequality and hence we lose $\d/2$. Our method allows to avoid such loses.

\section{Another proof of Theorem \ref{t:R_E_3_introduction} and further remarks}
\label{sec:proof1}

As was shown in \cite{KS2} that quantities $d^{+} (A)$, $d^\times (A)$ are bounded above by
$$
    d^{+}_* (A) := \min_{B \neq \emptyset} \frac{|AB|^2}{|A||B|}
        \quad \mbox{ and } \quad
    d^{\times}_* (A) := \min_{B \neq \emptyset} \frac{|A+B|^2}{|A||B|} \,,
$$
correspondingly, see also \cite{RR-NS}.
It turns out that there is a sum--product-type result involving just the quantities  $d^{+}_* (A)$, $d^{\times}_* (A)$ but not sum sets or product sets (which are hidden in the definitions of $d^{+}_* (A)$, $d^{\times}_* (A)$, nevertheless).

\begin{theorem}
    Let $A\subset \R$ be a finite set.
    Then
\begin{equation}\label{f:dd}
    |A| \lesssim  d^{+}_* (A) \cdot d^{\times}_* (A) \,.
\end{equation}
\label{t:dd}
\end{theorem}
\begin{proof}
    Applying Lemma \ref{l:d_r}, Lemma  \ref{l:D_q} and the H\"{o}lder inequality,
    we obtain for any nonempty finite set $B\subset \R$ that
$$
    \frac{|A|^2 |B|}{|A+B|^2}
        \le
            \frac{\E_3^{+}(A,-B)}{|A||B|^2} \le q^{+} (A) \lesssim D^{+} (A) \ll d^{+} (A) \le d^{+}_* (A) \,.
$$
    In other words, for any such $B$ one has
$$
    |A| \lesssim d^{+}_* (A) \cdot \frac{|A+B|^2}{|A||B|}
$$
    as required.
$\hfill\Box$
\end{proof}

\bigskip

Of course bound (\ref{f:dd}) is optimal up to logarithms. Actually, we have proved in Theorem \ref{t:dd} that
$|A| \lesssim D^{+} (A) \cdot d^{\times}_* (A)$ and $|A| \lesssim D^{\times} (A) \cdot d^{+}_* (A)$.

\bigskip

Theorem \ref{t:BW_D}, combining with Lemmas \ref{l:gen_sigma}, \ref{l:gen_sigma_m}, gives us an
analog of Theorem \ref{t:BW'}
(or one can repeat the arguments from \cite{KS2} directly, we left this for the interested reader).

\begin{corollary}
    Let $A\subset \R$ be a finite set and $\d = 1/5$.
    Then there are two disjoint subsets $B$ and $C$ of $A$ such that $A = B\sqcup C$ and
$$
    \max\{ \E^{+} (A,B), \E^{\times} (A,C)\} \lesssim  |A|^{3-\d} \,.
$$
\label{c:BW''}
\end{corollary}

Of course Theorem \ref{t:BW_D} and Lemmas \ref{l:gen_sigma}, \ref{l:gen_sigma_m}
allows to calculate the higher energies of the splitting sets $B,C$.
We give a sketch of a more direct proof in the case of $\E^+_3, \E^\times_3$ energies, using Proposition \ref{p:E_k-norm}.

\begin{theorem}
    Let $A\subset \R$ be a finite set and $\d_1 = 2/5$.
    Then there exists
    two disjoint subsets $B$ and $C$ of $A$ such that $A = B\sqcup C$ and
$$
    \max\{\E^{+}_3 (B), \E^{\times}_3 (C)\} \lesssim  |A|^{4-\d_1} \,,
$$
\label{t:R_E_3}
\end{theorem}
\begin{proof}{\bf (sketch)}
Using the arguments of the proof of Theorem 20 from \cite{KS2} one finds a set $A_1\subseteq A$ such that
$|A_1| \gtrsim (\E^\times_3 (A))^{1/2} |A|^{-1}$ and
$$
    |A_1|^5 |A|^2 \gtrsim \E^\times_3 (A) \E^+_3 (A_1)
$$
(and, similarly, a dual version).
After that applying  the notation and the algorithm of the proof of Corollary 21 of the paper
 or following the proof of Theorem \ref{t:BW_D} as well as Proposition \ref{p:E_k-norm},
 we obtain
$$
    \E^{+}_3 (D_j) \lesssim M |A|^{-2} |D_j|^{5} \,,
$$
where
\begin{equation}\label{f:D_j_low}
    |D_j| \gtrsim |A| M^{-1/2} \,,
\end{equation}
and
with help of
the H\"{o}lder inequality
$$
    (\E^{+}_3 (B_k))^{1/6} = (\E^{+}_3 (D_1 \cup \dots \cup D_k))^{1/6} \le \sum_{j=1}^k (\E^{+}_3 (D_j))^{1/6}
        \lesssim (M |A|^{-2})^{1/6} \sum_{j=1}^k |D_j|^{5/6}
        \le
$$
$$
        \le
            (M |A|^{-2})^{1/6} |A|^{5/6} k^{1/6}
                        \lesssim
                            M^{1/4} |A|^{1/2} \,.
$$
The last bound is a consequence of (\ref{f:D_j_low}), namely, $k \lesssim M^{1/2}$.
Hence
$$
    \E^{+}_3 (B_k) \lesssim M^{3/2} |A|^{3} \,.
$$
Optimizing over $M$, that is solving the equation $M^{3/2} |A|^3 = |A|^4/M$ and choosing  $M=|A|^{2/5}$, we obtain
the result.
$\hfill\Box$
\end{proof}

\bigskip

We do not consider the situation of higher energies (although Proposition \ref{p:E_k-norm} allows to do it) because they will not so effective. The fact is the Szemer\'{e}di--Trotter theorem naturally corresponds to $\E^{+}_3 (A)$, $\E^\times_3 (A)$ energies.

\bigskip

In \cite{Shkredov_R[A]} the author considered  a
more general context than usual sum--product setting.
The method, combining with the arguments of \cite{KS2}, allows to obtain a variant of Theorem \ref{t:BW'}, in particular.

\begin{theorem}
    Let $A\subset \R$ be a finite set, $\a \neq 0$ be a real number, and $\d = 1/5$.
    Then there are two disjoint subsets $B$ and $C$ of $A$ such that $A = B\sqcup C$ and
$$
    \max\{ \E^{\times} (B), \E^{\times} (C+\a)\} \lesssim  |A|^{3-\d} \,.
$$
    Further, there are disjoint subsets $B'$ and $C'$ of $A$ such that $A = B'\sqcup C'$ and
$$
    \max\{ \E^{+} (B'), \E^{+} (\a/ C')\} \lesssim  |A|^{3-\d} \,.
$$
\label{t:BW_mult}
\end{theorem}

Again one can prove a similar result for the energies $\E_3^{+} (B)$, $\E_3^\times (C+\a)$
 or for the quantities $D^{+} (B)$, $D^\times (C+\a)$ but we left it for the interested reader.
Let us derive a consequence of the result above.

\begin{corollary}
    Let $A\subset \R$ be a finite set.
    Put
$$
    R[A] = \left\{ \frac{a_1-a}{a_2-a} ~:~ a,a_1,a_2 \in A,\, a_2 \neq a \right\} \,.
$$
    Then there are two sets $R', R'' \subseteq R[A]$, $|R'|, |R''| \ge |R[A]|/2$
    such that $\E^\times (R') \lesssim |R'|^{3-1/5}$ and $\E^+ (R'') \lesssim |R''|^{3-1/5}$.
\label{c:R_energy}
\end{corollary}
\begin{proof}
First of all let us prove the existence of the set $R'$.
Put $R=R[A]$, $R^* = R \setminus \{0\}$, and $\d =1/5$.
Using Theorem \ref{t:BW_mult}, we find $B,C \subseteq R$ such that $R = B\sqcup C$ and
$$
    \max\{ \E^{\times} (B), \E^{\times} (C-1)\} \lesssim  |R|^{3-\d} \,.
$$
If $|B|\ge |R|/2$ then we are done.
Suppose not.
Then $|C| \ge |R|/2$ and in view of the formula $R=1-R$, see \cite{Shkredov_R[A]}, we obtain that
$C' := 1-C \subseteq R$, $|C'| = |C| \ge |R|/2$ and
$$
    \E^{\times} (C') = \E^{\times} (1-C) = \E^\times (C-1) \lesssim  |R|^{3-\d} \,.
$$
So, putting $R'$ equals $B$ or $C'$, we obtain the result.

To find the set $R''$ note that $(R^*)^{-1} = R^*$ and use the second part of Theorem \ref{t:BW_mult} as well as the arguments above.
This completes the proof.
$\hfill\Box$
\end{proof}

\bigskip

In particular,
Corollary \ref{c:R_energy}
says that the set $R$ has large $R+R$ and $RR$ (the last fact is known from paper \cite{Shkredov_R[A]} or can be obtained as a direct  application of the Szemer\'{e}di--Trotter theorem).



\bigskip

The same proof allows us to find a subset $A'_s$ of the set $A_s \cup (A_{-s}) = A_s \cup (A_s -s)$, $A_s = A\cap (A+s)$,
$s\in (A-A) \setminus \{0\}$ of cardinality $|A_s|/2$ such that $\E^\times (A'_s) \lesssim |A'_s|^{3-1/5}$
(similarly one can consider the set $A\cap (s-A)$
and find a subset of size at least $|A\cap (s-A)|/2$ with small multiplicative energy).
This question is a dual one which appeared in \cite{KS1}, \cite{KS2}.
The same result holds for some multiplicative analog of the sets $A_s$, namely, $A^*_s = A\cap (s/A)$,
$s\in AA \setminus \{0\}$. 
\section{Appendix}
\label{sec:appendix}

We finish the paper
discussing some generalizations of norms $E_k$.
Because our arguments almost repeat the
methods
of section \ref{sec:proof} we give the sketch of the proofs sometimes.

Take $l \ge 2$, $k\ge 2$ and suppose that either $k$ or $l$ is even.
Basically, we restrict ourselves considering the case of real functions.
For any such a function $f$ put
\begin{equation}\label{f:E_k_l}
    \| f \|^{kl}_{E_{k,l}} := \sum_{x_1,\dots,x_{k-1}} \Cf^l_{k} (f) (x_1,\dots,x_{k-1})
        =
            \sum_{y_1,\dots, y_{l-1}} \Cf^k_l (f) (y_1,\dots,y_{l-1}) \ge 0 \,,
\end{equation}
where we have used formula (\ref{f:gen_C}) of Lemma \ref{l:commutative_C} to obtain  the second identity in (\ref{f:E_k_l}).
Again for even $k$ and $l$, we
get
$\|f \|_{E_{k,l}} \ge \| f\|_l, \|f\|_k$ and hence $\| f \|_{E_{k,l}} = 0$ iff $f\equiv 0$ in the case.

\bigskip

Similarly, we obtain  an analog of Lemma \ref{l:E_k_prod}.

\begin{lemma}
    For any $k,l\ge 2$ and arbitrary functions $\_phi_1, \dots, \_phi_k : \Gr^{l} \to \C$,
    $\_phi_j = (\_phi^{(1)}_j, \dots, \_phi^{(l)}_j)$ the following holds
\begin{equation}\label{f:E_k_l_prod}
    \left| \sum_{x \in \Gr^{l-1}} \Cf_{l} (\_phi_1) (x) \dots \Cf_{l} (\_phi_k) (x) \right|
        \le
            \prod_{j=1}^{k} \prod_{i=1}^l \| |\_phi^{(i)}_j| \|_{E_{k,l}} \,.
\end{equation}
    If $k,l\ge 2$ are even and all functions are real then one can remove modulus from formula (\ref{f:E_k_l_prod}).
\label{l:E_k_l_prod}
\end{lemma}
\begin{proof}
    Let $\_phi = (\_phi^{(1)}, \dots, \_phi^{(l)})$.
    By the H\"{o}lder inequality it is sufficient to have deal with
$$
    \sum_{x \in \Gr^{l-1}} \Cf^k_{l} (\_phi) (x)
        = \sum_{x\in \Gr^{k-1}} \Cf_k (\_phi^{(1)}) (x) \dots \Cf_k (\_phi^{(l)}) (x) \,,
$$
    where we have used formula (\ref{f:gen_C}) of Lemma \ref{l:commutative_C}.
   Applying
   the H\"{o}lder inequality again we obtain the required result.
$\hfill\Box$
\end{proof}

\bigskip

An analog of Proposition \ref{p:E_k-norm} is the following.

\begin{proposition}
    Let $k,l \ge 2$ be integers.
    Then for any pair of functions $f,g : \Gr \to \C$ the following holds
    $$\| f + g\|_{E_{k,l}} \le \| |f| \|_{E_{k,l}} + \| |g| \|_{E_{k,l}} \,.$$
    If
    we consider just real functions
    and $k,l$ are  even
    numbers   then
    $\| \cdot \|_{E_{k,l}}$ is a norm.
\label{p:E_k_l-norm}
\end{proposition}
\begin{proof}
    Recall that $\Omega_l := \{0,1\}^l$.
We have
$$
    \sigma:= \| f + g\|^{kl}_{E_{k,l}} = \sum_{x\in \Gr^{l-1}} \Cf^k_l (f+g) (x)
        =
            \sum_{x\in \Gr^{l-1}} \left( \sum_{\eps \in \Omega_l} \Cf_l (\_phi_\eps) (x) \right)^k
                =
$$
$$
                =
                    \sum_{\sum_{\eps \in \Omega_l} n_\eps = k}
                        \binom{k}{n_1,\dots,n_{2^l}} \sum_{x\in \Gr^{l-1}} \prod_{\eps \in \Omega_l} \Cf^{n_\eps}_l (\_phi_\eps) (x) \,,
$$
where for $\eps = (\eps_1,\dots, \eps_{l}) \in \Omega_l$ we put $\_phi_\eps = (\_phi_{\eps_1}, \dots, \_phi_{\eps_l})$
and denote $\_phi_1 = f$, $\_phi_0 = g$.
Let
$n=(n_\eps) \in \N_0^{2^l}$ and $q(n) := \sum_{\eps \in \Omega_l} \wt (\eps) n_\eps$.
Applying Lemma \ref{l:E_k_l_prod} and Lemma \ref{l:comb}, we obtain
$$
    \sigma
        \le
            \sum_{\sum_{\eps \in \Omega_l} n_\eps = k} \binom{k}{n_1,\dots,n_{2^l}}
                \| |f| \|^{q(n)}_{E_{k,l}} \| |g| \|^{kl-q(n)}_{E_{k,l}}
                    =
                        \sum_{i+j=kl}  \frac{(kl)!}{i! j!} \| |f| \|^{i}_{E_{k,l}} \| |g| \|^{j}_{E_{k,l}}
                            =
$$
$$
                            =
                                (\| |f| \|_{E_{k,l}} + \| |g| \|_{E_{k,l}})^{kl}
$$
as required.
$\hfill\Box$
\end{proof}

\bigskip

\noindent{I.D.~Shkredov\\
Steklov Mathematical Institute,\\
ul. Gubkina, 8, Moscow, Russia, 119991}
\\
and
\\
IITP RAS,  \\
Bolshoy Karetny per. 19, Moscow, Russia, 127994\\
{\tt ilya.shkredov@gmail.com}


\begin{thebibliography}{99}



\bibitem{BW}
{\sc A.~Balog, T.D.~Wooley, }
\emph{A low--energy decomposition theorem, } arXiv:1510.03309v1 [math.NT] 12 Oct 2015.


\bibitem{ES}
{\sc P.~Erd\"{o}s, E.~Szemer\'{e}di, }
\emph{On sums and products of integers, }
Studies in pure mathematics, 213--218, Birkh\"auser, Basel, 1983.


\bibitem{Hanson}
{\sc B. Hanson, }
\emph{Estimates for character sums with various convolutions, } arXiv:1509.04354.


\bibitem{KS1}
{\sc S.V.~Konyagin, I.D.~Shkredov, }
{\em On sum sets of sets, having small product sets, }
Transactions of Steklov Mathematical Institute, {\bf 3}:290 (2015), 304--316.


\bibitem{KS2}
{\sc S.V.~Konyagin, I.D.~Shkredov, }
{\em New results on sum--product in $\R$, }
Transactions of Steklov Mathematical Institute, accepted,
arXiv:1602.03473v1 [math.CO] 10 Feb 2016.



\bibitem{RR-NS}
{\sc O.E.~Raz, O.~Roche--Newton, M.~Sharir, }
{\em Sets with few distinct distances do not have heavy lines, }
arXiv:1410.1654v1 [math.CO] 7 Oct 2014.


\bibitem{Rudin_book}
{\sc W.~Rudin, } {\em Fourier analysis on groups,}  Wiley 1990 (reprint of the 1962 original).


\bibitem{SS1}
{\sc T. Schoen, I.D. Shkredov, }
{\em Higher moments of convolutions, }
J. Number Theory {\bf 133} (2013), no. 5, 1693--1737.



\bibitem{s_energy}
{\sc I.D. Shkredov, }
{\em Energies and structure of additive sets, }
Electronic Journal of Combinatorics, 21(3) (2014), \#P3.44, 1--53.


\bibitem{s_sumsets}
{\sc I.D. Shkredov, }
{\em On sums of Szemer\'{e}di--Trotter sets, }
Transactions of Steklov Mathematical Institute, 289 (2015), 300--309.


\bibitem{Shkredov_R[A]}
{\sc I.D.~Shkredov, }
{\em Difference sets are not multiplicatively closed, }
arXiv:1602.02360v2 [math.NT] 14 Feb 2016.


\bibitem{soly}
{\sc J. Solymosi, }
{\em Bounding multiplicative energy by the sumset, }
Advances in Mathematics Volume 222, Issue 2, (2009), 402--408.



\bibitem{TV}
{\sc T. Tao, V. Vu, }
{\em Additive Combinatorics, }
Cambridge University Press (2006).



\end{thebibliography}
\end{document}